\documentclass[11pt,reqno]{amsart}  %final

\title{Graph properties, graph limits and entropy}

\date{12 December, 2013}  

\author{Hamed Hatami}
\address{School of Computer Science, McGill University, Montreal, Canada.}
\email{hatami@cs.mcgill.ca}
\author{Svante Janson}
  \address{Department of Mathematics, Uppsala University, 
PO Box 480, SE-751 06 Uppsala, Sweden.}
  \email{svante.janson@math.uu.se}
\author{Bal{\'a}zs Szegedy}
\address{University of Toronto, Department of Mathematics, St George St. 40,
Toronto, ON,
M5R 2E4, Canada.}
\email{szegedy@math.toronto.edu}
\thanks{%
HH supported by an NSERC, and an FQRNT grant.
SJ  supported by the Knut and Alice Wallenberg Foundation.
This research was mainly done during the workshop 
\emph{Graph limits, homomorphisms and structures II} at
Hrani{\v c}n{\'i} Z{\'a}me{\v c}ek,  Czech Republic, 2012.}

\usepackage{amsmath,amssymb,amsthm,amscd,verbatim}
\usepackage{color}
\usepackage{lscape}
\usepackage{ifdraft}
\usepackage{stmaryrd}
\usepackage{epic,eepic}
\usepackage[dvips]{graphicx}
\usepackage[T1]{fontenc}

\numberwithin{equation}{section}
\usepackage{textcomp}
\usepackage{amsbsy}

\newtheorem{theorem}{Theorem}[section]

\newtheorem{lemma}[theorem]{Lemma}

\newtheorem{corollary}[theorem]{Corollary}

\theoremstyle{definition}
\newtheorem{example}[theorem]{Example}

\theoremstyle{remark}
\newtheorem{remark}[theorem]{Remark}

\newcommand{\Ent}{{\operatorname{Ent}}}
\newcommand{\rand}{{\operatorname{rand}}}
\newcommand{\supp}{{\operatorname{supp}}}
\def\Ex {{\mathbb E}}
\def\PP{\mathcal{P}}
\def\BB{\mathcal{B}}
\def\QQ{\mathcal{Q}}
\def\UU{\mathcal{U}}
\def\LL{\mathcal{L}}

\newenvironment{romenumerate}[1][0pt]{% optional argument changes indentation
\addtolength{\leftmargini}{#1}\begin{enumerate}% gives (i), (ii) etc.
 }{\end{enumerate}}

\newcounter{oldenumi}
{\setcounter{oldenumi}{\value{enumi}}
\begin{romenumerate} \setcounter{enumi}{\value{oldenumi}}}
{\end{romenumerate}}

\newcounter{thmenumerate}

\newcounter{romxenumerate}   %less indented than standard.

\newcounter{xenumerate}   %no left indentation; thus wider lines

\newcommand{\refT}[1]{Theorem~\ref{#1}}

\newcommand{\refL}[1]{Lemma~\ref{#1}}
\newcommand{\refR}[1]{Remark~\ref{#1}}
\newcommand{\refS}[1]{Section~\ref{#1}}

\newcommand{\refE}[1]{Example~\ref{#1}}

\newcommand{\refand}[2]{\ref{#1} and~\ref{#2}}

\begingroup
  \divide\count255 by 60
  \multiply\count255 by -60
  \advance\count255 by \time
  \ifnum \count255 < 10 \xdef\klockan{\the\count1.0\the\count255}
  \else\xdef\klockan{\the\count1.\the\count255}\fi
\endgroup

\newcommand\set[1]{\ensuremath{\{#1\}}}
\newcommand\bigset[1]{\ensuremath{\bigl\{#1\bigr\}}}

\newcommand\lrset[1]{\ensuremath{\left\{#1\right\}}}

\newcommand\bigpar[1]{\bigl(#1\bigr)}
\newcommand\Bigpar[1]{\Bigl(#1\Bigr)}

\newcommand\lrpar[1]{\left(#1\right)}

\newcommand\bigabs[1]{\bigl|#1\bigr|}

\def\rompar(#1){\textup(#1\textup)}    % usage: \rompar(...)
\newcommand\xfrac[2]{#1/#2}

\def\xexp(#1){e^{#1}}

\newcommand\floor[1]{\lfloor#1\rfloor}

\newcommand\frax[1]{\{#1\}}

\newcommand\ntoo{\ensuremath{{n\to\infty}}}

\newcommand\mtoo{\ensuremath{{m\to\infty}}}

\newcommand\norm[1]{\|#1\|}

\newcommand\punkt{.\spacefactor=1000}    % om problem!
\newcommand\iid{i.i.d\punkt}    
\newcommand\ie{i.e\punkt}
\newcommand\eg{e.g\punkt}

\newcommand\cf{cf\punkt}

\newcommand{\aex}{a.e\punkt}
\newcommand{\aexx}{\aex}

\newcounter{CC}
\newcounter{cc}
\newcommand\E{\operatorname{\mathbb E{}}}

\newcommand\gd{\delta}

\newcommand\gG{\Gamma}

\newcommand\gs{\sigma}

\newcommand\eps{\varepsilon}

\newcommand\cC{\mathcal C}

\newcommand\cL{{\mathcal L}}

\newcommand\cP{\mathcal P}
\newcommand\cQ{\mathcal Q}

\newcommand\cU{{\mathcal U}}

\newcommand\cW{\mathcal W}

\newcommand\etta{\boldsymbol1}

\renewcommand{\=}{:=}

\newcommand\oi{[0,1]}

\newcommand\dd{\,\mathrm{d}}

\newcommand\lhs{left-hand side}
\newcommand\rhs{right-hand side}

\newcommand\logb{\log_2}
\newcommand{\tind}{t_{\mathrm{ind}}}
\newcommand{\cuq}{\overline{\cU}}
\newcommand{\cqq}{\overline{\cQ}}
\newcommand{\cuoo}{\widehat{\cU}}
\newcommand{\cqoo}{\widehat{\cQ}}
\newcommand\vv[1]{|#1|}
\newcommand\ql{\cQ^L}
\newcommand\W{\gG}
\newcommand\tG{\widetilde G}
\newcommand\HH{\Ent}
\renewcommand\Pr{\operatorname{\mathbb P{}}}
\renewcommand\P{\Pr}
\newcommand\hx{h^*}
\newcommand\UB{B_\gd}
\newcommand\wwrs{W^*_{r,s}}
\newcommand\wwx[1]{W^*_{#1}}
\newcommand\cPP{\overline\cP}
\newcommand\bW{\overline W}
\newcommand\gdcut{\delta_\square}
\newcommand\dcut{d_\square}
\newcommand\hN{\widehat{N}}
\newcommand\wxxx{W^{**}}
\newcommand\wxxxa{W^{**}_a}
\begin{document}

\begin{abstract}
We study the relation between the growth rate of a graph property and the
entropy of the
graph limits that arise from graphs with that property.
In particular, for hereditary classes we obtain a new description of the
colouring number, which by well-known results describes the rate of growth.

We study also random graphs and their entropies. We show, for example, that
if a hereditary property has a unique limiting graphon with maximal 
entropy, then a random graph with this property, selected uniformly at
random from all such graphs with a given order, converges to this maximizing
graphon  as the order tends to infinity.
\end{abstract}

\maketitle

\section{Introduction and results} \label{sec:intro}
In recent years a theory of convergent sequences of dense graphs has been
developed, see \eg{} the book \cite{LovaszBook}.  One can 
construct a limit object for  such a sequence in the form of certain symmetric measurable functions called 
graphons. The theory of graph limits not only provides a framework for addressing some previously unapproachable questions, but also 
leads to new interesting questions. For example one can ask: Which graphons arise as limits of sequences of graphs with a given property? 
Does a  sequence of random graphs drawn from the set of graphs with a given
property converge, and if so, what is the limit graphon? 
These types of questions has been studied for certain
properties~\cite{ChatterjeeDiaconis,MR2274085,MR2573956,SJ255}. In this
article we 
study the relation between these questions, 
the entropy of graphons, and the growth rate of graph properties. 

The growth rate of graph properties has been studied extensively 
in the past, see \eg{}
\cite{Alekseev,MR1425205,
MR1769217,MR1822715,
MR2047528,
MR2252788,
MR2467814}.
The standard method 
has been to use the Szemer\'edi regularity lemma, while 
we use graph limits; this should not be surprising, since
it has been known since the introduction of graph limits that there is a
strong connection with the Szemer\'edi regularity lemma.
Some of our proofs reminisce the proofs from previous works, but in
different formulations, 
\cf{} \eg{}  Bollob\'as and Thomason~\cite{MR1425205}.

\subsection{Preliminaries}
For every natural number $n$, denote $[n]:=\{1,\ldots, n\}$. In this paper all
graphs are simple and finite. For a graph $G$, let $V(G)$ and $E(G)$,
respectively, denote the vertex set and the edge set of $G$. 
We write for convenience $|G|$ for $|V(G)|$, the number of vertices.
Let $\mathcal{U}$ denote set of all unlabelled graphs.
(These are formally defined as equivalence classes of graphs up to
isomorphisms.) Moreover for $n \ge 1$, let $\mathcal{U}_n \subset
\mathcal{U}$ denote the set of all graphs in $\mathcal{U}$ with exactly $n$
vertices.   
Sometimes we shall work with labelled graphs.  For every $n \ge 1$, denote by
$\LL_n$ the set of all graphs with vertex set $[n]$.

We recall the basic notions of graph limits, see \eg{}
\cite{MR2274085,BCLSV1,SJ209,LovaszBook} for further details.
The \emph{homomorphism density} of a graph $H$ in a graph $G$, denoted by
$t(H;G)$, is the 
probability that a uniformly random mapping $\phi:V(H) \to V(G)$ preserves
adjacencies, 
\ie{} $uv \in E(H) \implies \phi(u)\phi(v) \in E(G)$.
The \emph{induced density} of a graph $H$ in a graph $G$, denoted by
$p(H;G)$, is the 
probability that a uniformly random \emph{embedding} of the vertices of $H$ in the
vertices of $G$ is an embedding of $H$ in $G$,
\ie{}  $uv \in E(H) \iff \phi(u)\phi(v) \in E(G)$.
(This is often denoted $\tind(H;G)$. We assume $\vv{H}\le\vv{G}$ so that
embeddings exist.) 
We call a sequence of finite
graphs $\{G_i\}_{i=1}^\infty$ 
with $\vv{G_i}\to\infty$
\emph{convergent} if for every finite graph $H$,
the sequence $\{p(H;G_i)\}_{i=1}^\infty$ converges. 
(This is equivalent to $\{t(H;G_i)\}_{i=1}^\infty$ being convergent for every
finite graph $H$.)
One then may construct a completion $\cuq$ of  $\mathcal{U}$ under this
notion of convergence.  
More precisely,  $\overline{\mathcal{U}}$ is a compact
metric space which contains $\mathcal{U}$ as a dense subset;
the functionals $t(H;G)$ and $p(H;G)$ extend by continuity to $G\in\cuq$, for
each fixed graph $H$;
elements of the
complement  $\cuoo := \overline{\mathcal{U}} \setminus \mathcal{U}$
are called \emph{graph limits};
a sequence of graphs $(G_n)$ converges
to a graph limit $\Gamma$ if and only if $\vv{G_n} \to \infty$ and $p(H;G_n) \to
p(H;\Gamma)$ for every graph $H$. 
Moreover a graph limit is uniquely determined
by the numbers $p(H;\Gamma)$ for all $H \in \mathcal{U}$.

It is shown in~\cite{MR2274085} that every graph limit $\Gamma$ can be
represented by a \emph{graphon}, which is a symmetric measurable function
$W:[0,1]^2 \to [0,1]$. The set of all graphons are denoted by
$\mathcal{W}_0$.
(We do not distinguish between graphons that are equal almost everywhere.)
Given a graph $G$ with vertex set $[n]$ and adjacency matrix $A_G$, we
define the corresponding graphon $W_G: [0,1]^2 \rightarrow \{0,1\}$ as follows. Let $W_G(x,y) :=
A_G(\lceil x n \rceil,\lceil y n \rceil)$ if $x,y \in (0,1]$, and if
$x=0$ or $y=0$, set $W_G$ to $0$. It is easy to see that if $(G_n)$ is a graph
sequence that converges to a graph limit $\Gamma$, then for every graph $H$,
\begin{equation*}
  \begin{split}
p(H;\Gamma)&= \lim_{n \to \infty} p(H;G_n)
\\&
=\lim_{n \to \infty} \Ex\left[\prod_{uv \in E(H)}
W_{G_n}(X_u,X_v) \prod_{uv \in E(H)^c} (1-W_{G_n}(X_u,X_v))\right],	
  \end{split}
\end{equation*}
where $\{X_u\}_{u \in V(H)}$ are independent random variables taking values in
$[0,1]$ uniformly, and $E(H)^c=\{uv: u \neq v, uv \not\in E(H)\}$. Lov\'asz
and Szegedy \cite{MR2274085}  
showed that for every graph limit $\Gamma$, there exists a graphon $W$ such
that for every graph $H$, we have $p(H;\Gamma)= p(H;W)$ where 
\begin{equation}
  \label{phw}
p(H;W):= \Ex\left[\prod_{uv \in E(H)} W(X_u,X_v) \prod_{uv
\in E(H)^c} (1-W(X_u,X_v))\right].
\end{equation}
Unfortunately, this graphon is not unique.
We say that two graphons $W$ and $W'$ 
are \emph{(weakly) equivalent} 
if they represent the same graph
limit, \ie, if $p(H;W)=p(H;W')$ for all graphs $H$.
For example, a graphon $W(x,y)$ is evidently equivalent to $W(\gs(x),\gs(y))$
for any measure-preserving map $\sigma:[0,1] \to [0,1]$. Not every pair of
equivalent graphons is related in this way, but almost:
Borgs, Chayes and Lov\'asz \cite{BCL:unique}
proved that 
if $W_1$ and $W_2$ are two different graphons representing the same graph
limit, then there exists a third graphon $W$ and
measure-preserving maps $\sigma_i:[0,1] \to [0,1]$, $i=1,2$, such that 
\begin{equation}
 \label{eq:graphon_uniquenss}
W_i(x,y) = W(\sigma_i(x),\sigma_i(y)),
\qquad\text{for \aexx{} }x,y.
\end{equation}
(For other characterizations of  equivalent graphons, see \eg{}
\cite{BRmetrics} and \cite{SJ249}.)

The set $\cuoo$ of graph limits is thus a quotient space of the set $\cW_0$
of graphons.
Nevertheless, we shall not always distinguish between graph limits and their
corresponding graphons; it is often convenient (and customary) to let
a graphon $W$ also denote the corresponding graph limit.
For example, we may write $G_n\to W$ when a sequence of
graphs $\set{G_n}$ converges to the graph limit determined by the graphon
$W$;
%, and $W\in\cqq$, where $\cQ\subset\cU$, if there is such a sequence with
%$G_n\in\cQ$; 
similarly we say that a sequence of graphons $W_n$ converges to
$W$ in $\cW_0$ if the corresponding sequence of graph limits converges in
$\cuq$. (This makes $\cW_0$ into a topological space that is compact but not
Hausdorff.)
 
For every $n \ge 1$, a graphon $W$ defines a random graph $G(n,W) \in\LL_{n}$: 
Let $X_1,\ldots,X_n$ be an \iid{} sequence of random
variables taking values uniformly in $[0,1]$. Given $X_1,\ldots,X_n$, let $ij$
be an edge with probability $W(X_i,X_j)$, independently for all pairs $(i,j)$
with $1 \le i < j \le n$. It follows that for every $H \in \LL_n$, 
\begin{equation}
  \label{gnw}
\Pr[G(n,W) = H] = p(H;W).
\end{equation}
The distribution of $G(n,W)$ is thus the same for two equivalent graphons,
so we may define $G(n,\gG)$ for a graph limit $\gG$; this is a random graph
that also can be defined by the analogous relation
$\Pr[G(n,\gG) = H] = p(H;\gG)$ for $H\in\LL_n$.

\subsection{Graph properties and entropy\label{sec:mainI}}

A subset of the set $\mathcal{U}$ is called a \emph{graph class}. Similarly a
\emph{graph property} is a property of graphs that is invariant under graph
isomorphisms. There is an obvious one-to-one correspondence between graph
classes and graph properties and we will not distinguish between a graph
property and the corresponding class. Let $\QQ \subseteq \mathcal{U}$ be
a graph class. For every $n\ge1$, we denote by $\QQ_n\=\QQ\cap\cU_n$ 
the set of graphs in $\QQ$
with exactly $n$ vertices. 
We also consider the corresponding class of labelled graphs, and define $\ql_n$
to be the set of all graphs in $\cL_n$ that belong to $\QQ$ (when we ignore
labels).
Furthermore, we let $\overline{\QQ} \subseteq \overline{\UU}$ be
the closure of $\QQ$ in $\overline{\UU}$ and 
$\cqoo\=\cqq\cap\cuoo=\overline{\QQ}\setminus\QQ$
the set of graph limits that are limits of sequences of graphs in $\QQ$.

Define the \emph{binary entropy} function $h:[0,1] \mapsto \mathbb{R}_+$ as
\begin{equation*}
h(x)=-x \logb(x) - (1-x) \logb(1-x)  
\end{equation*}
for $x \in \oi$, with the interpretation $h(0)=h(1)=0$ so that
$h$ is continuous on $[0,1]$,
 where here and throughout the paper $\logb$ denotes
the logarithm to the base $2$. 
Note that $0\le h(x)\le1$, with $h(x)=0$
attained at $x=0,1$
and $h(x)=1$ at $x=1/2$, only.
The \emph{entropy} of a graphon $W$ is defined as
\begin{equation}
  \label{entw}
\Ent(W) := \int_0^1 \int_0^1  h(W(x,y)) \dd x \dd y.
\end{equation}
This is related to the entropy of random graphs, see 
\cite{Aldous} and \cite[Appendix D.2]{SJ249} and \eqref{sjw} below;
it has also previously been used by 
Chatterjee and Varadhan \cite{ChatterjeeVaradhan} 
and Chatterjee and Diaconis \cite{ChatterjeeDiaconis} 
to study large deviations of random graphs and exponential models of random
graphs. 
Note that it follows from the uniqueness result (\ref{eq:graphon_uniquenss})
that the entropy is a function of the underlying graph limit and  
it does not depend on the choice of the graphon representing it;
we may thus define the entropy $\Ent(\gG)$ of a graph limit $\gG$ as the
entropy $\Ent(W)$ of any graphon representing it.

Our first theorem bounds the rate of growth of an arbitrary graph class in
terms of the entropy of the limiting graph limits (or graphons).  

\begin{theorem}
\label{thm:T2}
Let $\QQ$ be a class of graphs. Then
\begin{equation}
\label{eq:thm_T2}
\limsup_{n \to \infty} \frac{\logb |\QQ_n|}{\binom {n}2}  
\le \max_{\gG \in
\cqoo} \Ent(\gG).
\end{equation}
\end{theorem}

We present the proofs of this and the following theorems 
in Section~\ref{sec:proofT2}.

\begin{remark}\label{Rn!}
  For any graph class $\QQ$, and $n\ge1$,
  \begin{equation}\label{n!}
	|\QQ_n| \le|\ql_n|\le n!|\QQ_n|.
  \end{equation}
The factor $n!$ is for our purposes small and can be ignored, since $\logb
n!=o(n^2)$. Thus we may replace $|\QQ_n| $ by $|\ql_n|$ in \refT{thm:T2}.
The same holds for the theorems below.
\end{remark}

\begin{remark}\label{R1}
$|\cQ_n|\le|\cU_n|\le|\cL_n|=2^{\binom n2}$, so the \lhs{} of
  \eqref{eq:thm_T2}   
is at most 1, and it  equals 1 if $\cQ$ is the class of all
graphs, \cf{} \eqref{n!}.
Furthermore, by \eqref{entw}, $\Ent(W)\in\oi$ for every graphon $W$.
In the trivial case when $\cQ$ is a finite class, $\cQ_n=\emptyset$ for all
large $n$ and the \lhs{} is $-\infty$; in this case $\cqoo=\emptyset$ and
the \rhs{} is also (interpreted as) $-\infty$. 
We exclude in the sequel this trivial case;
thus both sides of   \eqref{eq:thm_T2} are in $\oi$.
Note further that $\Ent(W)=1$ only when $W=1/2$ \aex; thus the \rhs{} of 
  \eqref{eq:thm_T2} equals 1 if and only if $\cqoo$ contains the 
graph limit defined by the constant graphon $W=1/2$. (This graphon is
the limit of sequences of quasi-random graphs, see \cite{MR2274085}.)
\end{remark}

A graphon is called \emph{random-free}
if it is $\{0,1\}$-valued almost 
everywhere, see \cite{LovaszSzegedy:regular,SJ249}.
Note that a graphon 
$W$ is random-free if and only if $\Ent(W)=0$. 
This is preserved by equivalence of graphons, so we may define a graph limit to
be random-free if some (or any) representing graphon is random-free;
equivalently, if its entropy is 0.
A
property $\QQ$ is called \emph{random-free} if every $\gG \in \cqoo$ is
random-free. Theorem~\ref{thm:T2} has the following immediate
corollary:

\begin{corollary}
If $\QQ$ is  a random-free class of graphs, then $|\QQ_n|=2^{o(n^2)}$.
\end{corollary}

For further results on random-free graphons and random-free classes of graphs, see
Hatami and Norine \cite{HatamiNorine12}.

A graph class $\mathcal{P}$ is \emph{hereditary} if whenever a graph $G$
belongs to $\QQ$, then every induced subgraph of $G$  
also belongs to $\mathcal{P}$. 

Our second theorem says that when $\QQ$ is a hereditary graph property, equality
holds in (\ref{eq:thm_T2}).
(See also \refT{thm:TR} below.) 

\begin{theorem}\label{thm:T4}
Let $\QQ$ be a hereditary class of graphs. Then
$$
\lim_{n \to \infty} \frac{\logb |\QQ_n|}{\binom{n}{2}}  
= \max_{\gG \in\cqoo} \Ent(\gG).
$$
\end{theorem}

Our next theorem concerns the limit of the sequences of random graphs that  are
sampled from a graph class. There are two natural ways to sample a random 
graph sequence $(G_n)$, with $\vv{G_n}=n$,
from a graph class $\QQ$. 
The first  is to pick an unlabelled graph $G_n$ uniformly at random from
$\QQ_n$, for each $n\ge1$ (assuming that $\QQ_n\neq\emptyset$).  
The second is to pick a labelled graph $G_n$ uniformly at random from
$\ql_n$.
We call the resulting random graph
$G_n$ a \emph{uniformly random unlabelled element of $\QQ_n$}
and a \emph{uniformly random labelled element of $\QQ_n$}, respectively.

\begin{theorem}
\label{thm:T3}
Suppose that $\max_{\gG \in \cqoo} \Ent(\gG)$ is attained by a unique graph
limit $\gG_\QQ$.  Suppose further that equality holds in \eqref{eq:thm_T2}, 
\ie
\begin{equation}
\label{eq:thm_T3}
\lim_{n \to \infty} \frac{\logb |\QQ_n|}{\binom{n}{2}} =  \Ent(\gG_\QQ).
\end{equation}
Then
\begin{romenumerate}
 \item 
 If\/ $G_n \in \mathcal{U}_n$ is a uniformly random unlabelled
element of $\QQ_n$,
then $G_n$ converges to $\gG_\QQ$ in probability as \ntoo.
 \item 
The same holds
if\/ $G_n \in \LL_n$ is a uniformly random labelled element of $\ql_n$.
\end{romenumerate}
\end{theorem}
\begin{remark}
Note that for hereditary properties, it  suffices to only
assume that $\max_{\gG \in \cqoo} \Ent(\gG)$ is attained by a unique graph
limit $\gG_\QQ$ as then \eqref{eq:thm_T3} follows from Theorem~\ref{thm:T4}.
\end{remark}

The next theorem concerns sequences of random graphs drawn from arbitrary
distributions, not necessarily uniform.
A random labelled [unlabelled] graph $G_n$
on $n$ vertices is thus any random variable with values in $\cU_n$
[$\LL_n$].
We consider convergence in distribution of $G_n$, regarding $G_n$ as a
random element of $\cU\subset\cuq$ (ignoring labels if there are any); the
limit in distribution (if it exists) is thus a random element of $\cuq$,
which easily is seen to be concentrated on $\cuoo$; in other words, the
limit is a random graph limit.

Recall that the entropy $\HH(X)$ of a random variable $X$ taking values in some
finite (or countable) set $A$ 
is $\sum_{a\in A} (-p_a\logb p_a)$, where 
$p_a\=\P(X=a)$.

%Theorem TC
\begin{theorem}
\label{thm:TC}
Suppose that $G_n$ is a 
(labelled or unlabelled)
random graph on $n$ vertices with some distribution
$\mu_n$. Suppose further that as $n \to \infty$, $G_n$ converges in distribution
to some random graph limit with distribution $\mu$. Then
$$ \limsup_{n \to \infty}\frac{\HH(G_n)} {\binom{n}{2}}
\le \max_{W \in \supp(\mu)}\Ent(W),
$$
where $\supp(\mu)\subseteq\cuoo$ is the support of the probability measure
$\mu$.
\end{theorem}

\subsection{Maximal entropy graphons}\label{sec:mainII}
The results in Section~\ref{sec:mainI} show that graphons with maximal entropy
capture the growth rate and other asymptotic behaviors of graph classes. In
this section we study the structure 
of those graphons for hereditary classes. 

We define the 
%\emph{support} of a graphon $W$ as 
%$\bigset{(x,y) \in [0,1]^2 : W(x,y)>0}$, its
\emph{randomness support} of a graphon $W$ as 
\begin{equation}
  \label{rand}
\rand(W) := \left\{(x,y) \in [0,1]^2 : 0<W(x,y)<1 \right\}, 
\end{equation}
and its  \emph{random part} as the restriction of $W$ to $\rand(W)$. Finally
the  
\emph{randomness support graphon} of $W$ is defined as $\etta_{\rand(W)}$, the
indicator 
of its randomness support.

A graphon $W$ is called   $K_{r}$-free (where $r\ge1$)
if $p(K_{r},W)=0$; by \eqref{phw},
this is equivalent to
$\prod_{1 \le i < j \le r} W(x_i,x_j)=0$ for almost every $x_1,\ldots,x_r$.
(The case $r=1$ is trivial: no graphon is $K_1$-free.)
Recall that the Tur\'an graph $T_{n,r}$ is the balanced complete $r$-partite
graph with $n$ vertices. For each $r \ge 1$, the graphs $T_{n,r}$ converge to
the  $K_{r+1}$-free graphon $W_{K_r}$ as \ntoo.  

Let $E_{r}$ denote the support of $W_{K_r}$, 
\ie, $E_r\=\bigcup_{i\neq j}I_i\times I_j$ where $I_i\=((i-1)/r,i/r]$ for
  $i=1,\dots, r$,
and also define $E_{\infty} :=
[0,1]^2$.
For $1 \le r \le \infty$, let $R_r$ be the set of graphons $W$ such that
$W(x,y)=\frac{1}{2}$ on $E_{r}$ and $W(x,y)\in \{0,1\}$ otherwise. In other
words, $W$ has randomness support $E_{r}$ and its random part is $\frac{1}{2}$
everywhere. 
Note that $E_1=\emptyset$ and thus $R_1$ is the set of random-free graphons,
while $R_\infty$ consists only of the constant graphon $\frac{1}{2}$.
 If $W \in R_r$, then
\begin{equation}
\label{eq:TuranGraphons}
\Ent(W)= \int_{E_{r}} h(1/2)=|E_{r}| = 1 -\frac{1}{r}.
\end{equation}

A simple example of a graphon in $R_r$ is $\frac{1}{2} W_{K_r}$. 
(For $r<\infty$, this is the almost surely 
limit of a uniformly random subgraph of $T_{n,r}$ as $n \to \infty$.) More
generally, if $r < \infty$, we can modify $\frac{1}{2} W_{K_r}$ by changing
it on 
each square $I_i^2$ for $i=1,\ldots,r$ to a
symmetric measurable $\{0,1\}$-valued
function (\ie{} to any random-free graphon, scaled in the natural way);
this gives all graphons in $R_r$.

We let, for $1\le r<\infty$ and $0\le s\le r$, $\wwrs$ be the graphon in 
$R_r$ that is $1$ on $I_i\times I_i$ for $i\le s$ and 
$0$ on $I_i\times I_i$ for $i> s$. (Thus $\wwx{r,0}=\frac12W_{K_r}$.)

For a class $\QQ$ of graphs, let
$$\cqoo^* := \left\{\W \in \cqoo: \Ent(\W) = \max_{\W \in
\cqoo} \Ent(\W) \right\} $$
denote the set of graph limits in $\cqoo$ with maximum entropy. 
It follows from \refL{lem:EntSemicont} below that 
the maximum is attained and that 
$\cqoo^*$ is a 
non-empty
closed
subset of $\cqoo$, and thus a non-empty compact set.

After these preparations, we state the following result, improving 
\refT{thm:T4}. 
\begin{theorem}
\label{thm:TR}
Let $\QQ$ be a hereditary class of graphs. Then there exists a number 
$r \in\{1,2\ldots,\infty\}$ such that 
$\max_{\W \in \cqoo}\Ent(\W)=1-\frac{1}{r}$, 
every graph limit in $\cqoo^*$ can be
represented by a graphon $W \in R_r$, 
and
\begin{equation}
  \label{eq:thm_TR}
|\QQ_n|=2^{(1-r^{-1}+o(1))\binom{n}{2} }.
\end{equation}
Hence, $\cqoo^*=\cqoo\cap R_r$. 
Moreover, $r$ has the further characterisations
\begin{align}
r
&=\min \lrset{s\ge1: \mbox{$\etta_{\rand(W)}$ is
 $K_{s+1}$-free for all graphons $W \in \cqoo$}}
\label{eq:define_r}
\\
&=\sup\lrset{t:\wwx{t,u}\in\cqoo \text{ for some $u\le t$}},
\label{eq:r2}
\end{align}
where the minimum in \eqref{eq:define_r} is
interpreted as $\infty$ when there is no such $s$.
Furthermore $r=1$ if and only if $\QQ$ is random-free, and $r=\infty$ if and
only if $\QQ$ is the class of all graphs.
\end{theorem}

The result \eqref{eq:thm_TR} is a fundamental result for hereditary classes 
of graphs, proved by Alekseev \cite{Alekseev} and Bollob\'as and Thomason
\cite{MR1425205}, see also the survey \cite{MR2252788}
and \eg{} 
\cite{MR1769217,MR1822715,
MR2047528,
MR2504400,
MR2467814,
Alon.et.al-hereditary}.
The number $r$ is known as the \emph{colouring number} of  $\cQ$.

\begin{remark} \label{Rcrs}
 Let, for $1\le r<\infty$ and $0\le s\le r$, 
$\cC(r,s)$ be the hereditary class of all graphs such that the
  vertex set can be partitioned into $r$ (possibly empty) sets $V_i$ with
  the subgraph induced by $V_i$ complete for $1\le i\le s$ and empty for
  $s<i\le r$. Note that $G(n,\wwrs)\in\cC(r,s)$ a.s., and that every graph
  in $\cC(r,s)$ with $n$ vertices appears with positive probability.
(In fact, $G\in\cC(r,s)\iff p(G,\wwrs)>0$.) 
It follows from \eqref{gnw} and \refL{lem:hered} below that, for any
hereditary class $\cQ$, $\wwrs\in\cqoo$ if and only if
$\cC(r,s)\subseteq\cQ$.
Hence, \eqref{eq:r2} shows that $r$ (when finite) is the largest integer
such that $\cC(r,s)\subseteq\cQ$ for some $s$; this is the traditional
definition of the colouring number, see \eg{} \cite{MR2252788} where further
comments are given.
\end{remark}

\section{Examples}

We give a few examples to illustrate the results. We begin with a simple case.

\begin{example}[Bipartite graphs] \label{Ebipartite}
  Let $\cQ$ be the class of \emph{bipartite graphs}; note that this equals
  the class $\cC(2,0)$ in \refR{Rcrs}.
Suppose that a graph limit $\gG\in\cqoo$. Then there exists a sequence of
graphs $G_n\to\gG$ with $G_n\in\cQ$, 
where for simplicity we may assume $|G_n|=n$. 
Since $G_n$ is bipartite, it has a bipartition that can be assumed to be
\set{1,\dots,m_n} and \set{m_n+1,\dots,n}. By selecting a subsequence, we
may assume that $m_n/n\to a$ for some $a\in\oi$, and it is then easy to see
(for example by using the bipartite limit theory in \cite[Section 8]{MR2573956})
that $\gG$ can be represented by a graphon that vanishes on
$[0,a]^2\cup[a,1]^2$. Conversely, if $W$ is such a graphon, then the random
graph $G(n;W)$ is bipartite, and thus $W\in\cqoo$. Hence $\cqoo$ equals the
set of graph limits represented (non-uniquely) by the graphons 
\begin{equation}\label{bip}
 \bigcup_{a\in\oi} \bigset{W: W=0 \text{ on } [0,a]^2\cup[a,1]^2}.
\end{equation}
If $W$ is a graphon in the set \eqref{bip}, with a given $a$, then the
support of $W$ has measure at most $2a(1-a)$, and thus 
\begin{equation}
 \Ent(W)\le 2a(1-a),
\end{equation}
with equality if and only if $W=\frac12$ on 
$(0,a)\times(a,1)\cup (a,1)\times(0,a)$. The maximum entropy is obtained for
$a=1/2$, and thus
\begin{equation}
  \max_{\gG\in\cqoo}\Ent(\gG) =\tfrac12,
\end{equation}
and the maximum is attained by a unique graph limit, represented by the
graphon $\wwx{2,0}$ defined above. 

\refT{thm:T4} thus says that $|\cQ_n|=2^{\frac12\binom n2+o(n^2)}$
(which can be easily proved directly). 
\refT{thm:T3} says that if $G_n$ is a
uniformly random (labelled or unlabelled) bipartite graph, then $G_n\to
\wwx{2,0}$ in probability. 
The colouring number $r$ in \refT{thm:TR} equals 2, and
both \eqref{eq:define_r} and \eqref{eq:r2} are easily verified directly.
\end{example}

\begin{example}[Triangle-free graphs]
  Let $\cQ$ be the class of \emph{triangle-free} graphs.
It is easy to see that the corresponding class of graph limits $\cqoo$ is
the class of triangle-free graph limits $\set{\gG:p(K_3,\gG)=0}$ defined in
  \refS{sec:mainII}, see \cite[Example 4.3]{SJ255}. 

This class is strictly larger than the class of bipartite graphs; 
the set $\cqoo$ of triangle-free graph limits 
thus contains the set \eqref{bip} of bipartite graph limits, and it is
easily seen that it 
is strictly larger. (An example of a triangle-free graph limit that is not
bipartite is $W_{C_5}$.) 

We do not know any representation of all
triangle-free graph limits similar to \eqref{bip}, but it is easy to find
the ones of maximum entropy.
If a graphon $W$ is
  triangle-free, then so is its randomness support graphon, and
  \refL{lem:entropy_cliquefree} below shows that $\Ent(W)\le \frac12$, with
  equality only if $W\in R_2$ (up to equivalence). Furthermore, it is easy
  to see that if $W\in R_r$ is triangle-free, then $W(x,y)\neq 1$ a.e., and
  thus $W=\wwx{2,0}$. (Use \refT{thm:ErdosSimonovits} below, or note that
$\max\set{W(x,y),\frac12}$ is another triangle-free graphon.) 
Thus, as in \refE{Ebipartite}, $\wwx{2,0}$ represents the unique graph limit in
$\cqoo$ with maximum entropy.

\refT{thm:T4} and \ref{thm:TR} 
thus say that $|\cQ_n|=2^{\frac12\binom n2+o(n^2)}$,
as shown by Erd{\H o}s, Kleitman and Rothschild \cite{EKR76}.
(They also
proved that almost all triangle-free graphs are bipartite; this seems
related to the fact that the two graph classes have the same maximum entropy
graph limit, although we do not know any direct implication.)

\refT{thm:T3} says that if $G_n$ is a
uniformly random (labelled or unlabelled) triangle-free graph, then $G_n\to
\wwx{2,0}$ in probability. 

The same argument applies to $K_t$-free graphs, for any $t\ge2$. The
colouring number is $t-1$ and thus the number of such graphs of order $n$ is
$2^{\frac{r-2}{r-1}\binom n2+o(n^2)}$, as shown in \cite{EKR76}.
(See also
%More precisely, almost all $K_t$-free graphs are $(t-1)$-partite, see
%Kolaitis, Pr{\"o}mel and Rothschild 
\cite{KPR85,KPR87}.)
The unique graph limit of maximum entropy is
represented by $\wwx{t-1,0}$.
Thus \refT{thm:T3} applies and shows that, hardly
surprising, a random $K_t$-free graph converges (in probability) to the
graphon $\wwx{t-1,0}$. 
\end{example}

\begin{example}[Split graphs]
Another simple application of \refT{thm:T3} is given in 
\cite[Section  10]{SJ255}, where it is shown that the class of \emph{split
  graphs} has a unique graph limit with maximal entropy, represented by the
graphon 
$\wwx{2,1}$; this is thus the limit (in probability) of a uniformly  random
split graph.  
Recall that the class of split graphs equals $\cC(2,1)$ in \refR{Rcrs}; in
other words,
a graph is a split graph if its vertex set can be partitioned
into two sets, one of which is a clique and the other one is an isolated
set. 
(Equivalently, $G$ is a split graph if and only if $p(G;\wwx{2,1})>0$.)
\end{example}

Our final example is more complicated, and we have less complete results.

\begin{example}[String graphs]
  A \emph{string graph} is the intersection graph of a family of curves in
  the plane. In other words, $G$ is a string graph if there exists a
  collection \set{A_v:v\in V(G)} of curves such that 
$ij\in E(G)\iff A_i\cap A_j\neq\emptyset$.
It is easily seen that we obtain the same class of graphs if we allow the
sets $A_v$ to be arbitrary arcwise connected sets in the plane.

It is shown by Pach and T\'oth \cite{PachToth06} that the number of string
graphs of order $n$ is $2^{\frac34\binom n2+o(n^2)}$. 
Thus, Theorems \ref{thm:T4}
and \ref{thm:TR} hold with maximum entropy $\frac34$ and colouring number
$4$.

We study this further by interpreting the proof of \cite{PachToth06} in our
graph limit context.
To show a lower bound on the number of string graphs, \cite{PachToth06}
shows that every graph in the class $\cC(4,4)$ is a 
string graph. 
(This was proved already in \cite[Corollary 2.7]{Kratochvil86}.)
A minor modification of their construction is as follows:
Let $G$ be a graph with a partition $V(G)=\bigcup_{i=1}^4 V_i$ such that
each $V_i$ is a complete subgraph of $G$. Consider a drawing of the graph
$K_4$ in the plane, with vertices $x_1,...,x_4$
and non-crossing edges. Replace each edge $ij$ in $K_4$
by a number of parallel curves $\gamma_{vw}$ from $x_i$ to $x_j$, indexed
by pairs $(v,w)\in V_i\times V_j$. (All curves still non-intersecting except
at the end-points.) Choose a point $x_{vw}$ on each curve $\gamma_{vw}$, and
split $\gamma_{vw}$ into the parts $\gamma^*_{vw}$ from $x_i$ to $x_{vw}$  and
$\gamma^*_{wv}$ from  $x_{vw}$ to $x_j$, with $x_{vw}$ included in both parts.
If $v$ is a vertex in $G$, and $v\in V_i$, let $A_v$ be the (arcwise
connected) set consisting of $x_i$ and the curves $\gamma^*_{vw}$ for all 
$w\notin V_i$ such that $vw\in E(G)$. Then $G$ is the intersection graph
defined by the collection \set{A_v}, and thus $G$ is a string graph.

It follows, see \refR{Rcrs}, 
that if $\cQ$ is the class of string graphs, then $\wwx{4,4}\in\cqoo$.

To show an upper bound, Pach and T\'oth
\cite{PachToth06} consider the graph $G_5$, which
is the intersection graph of the family of the 15 subsets of order 1 or 2 of
\set{1,\dots,5}. They show that $G_5$ is not a string graph, but that
$G_5\in\cC(5,s)$ for every $0\le s\le 5$. Thus $\cC(5,s)\not\subseteq \cQ$,
and thus $\wwx{5,s}\notin\cqoo$, see \refR{Rcrs}. 

Consequently, we have $\wwx{4,4}\in\cqoo$ but $\wwx{5,s}\notin\cqoo$, for
all $s$. Hence 
\refT{thm:TR} shows
that the colouring number $r=4$, see
\eqref{eq:r2}, and that $\wwx{4,4}$ is one graphon in $\cqoo$ with maximal
entropy. 

However, in this case the graph limit of maximal entropy is \emph{not} unique.
Indeed, the construction above of string graphs works for any planar graph
$H$ instead of $K_4$, and $G$ such that its vertex set can be partitioned
into cliques $V_i$, $i\in V(H)$, with no edges in $G$ between $V_i$
and $V_j$ unless $ij\in E(H)$. 
(See \cite[Theorem 2.3]{Kratochvil86}.)
Taking $H$ to be $K_5$ minus an edge, we thus
see that if $G\in\cC(4,4)$, and we replace the clique on $V_1$ by 
a disjoint union of two cliques
(on the same vertex set $V_1$, leaving all other
edges), 
then the new graph is also a string graph. 
It follows by taking the limit of a suitable sequence of such graphs, 
or by \refL{lem:hered} below,
that if 
$I_i\=((i-1)/4,i/4]$ and $I_1$ is split into
  $I_{11}\=(0,a]$ and $I_{12}\=(a,1/4]$, where 
$0\le a\le 1/8$, 
then the graphon  $\wxxxa\in R_4$ 
obtained from $\wwx{4,4}$ by replacing the 
value 1 by $0$ on $(I_{11}\times I_{12})\cup(I_{12}\times I_{11})$
satisfies  $\wxxxa\in\cqoo\cap R_4=\cqoo^*$. 
Explicitly, 
\begin{equation*}
\wxxxa(x,y)=
\begin{cases}
1/2 & \text{on } \bigcup_{i\neq j}(I_i\times I_j);\\
0  & \text{on } (I_{11}\times I_{12}) \cup (I_{12}\times I_{11}); \\
1  & \text{on }  (I_{11}\times I_{11}) \cup(I_{12}\times I_{12}) 
 \cup \bigcup_{i=2}^4 (I_i\times I_i).
\end{cases}
	  \end{equation*}
Thus $\wxxx_0=\wwx{4,4}$, but the graphons
$\wxxx_a$ for $a\in[0,1/8]$ are not equivalent, for example because they
have different edge densities 
\begin{equation*}
\iint \wxxx_a=\frac58-\frac a2+2a^2=\frac{19}{32}+2\Bigpar{\frac18-a}^2.  
\end{equation*}
Thus there are infinitely many graph limits in $\cqoo^*=\cqoo\cap R_4$. 
(We do not know
whether there are further such graph limits.)

Consequently, \refT{thm:T3} does not apply to string graphs. We do not know
whether a uniformly random string graph converges (in probability)
to some graph limit as the size tends to infinity, and if so, what the limit
is. We leave this as an open problem.
\end{example}

\section{Some auxiliary facts}
We start by recalling some basic facts about the binary entropy.  First note
that $h$ is concave on $[0,1]$. In particular if $0 \le x_1 \le x_2 \le 1$, then
$$h(x_2)-h(x_1) \le h(x_2-x_1) - h(0) =h(x_2-x_1),$$
and
$$-(h(x_2)-h(x_1))= h(x_1)-h(x_2) = h(1-x_1) - h(1-x_2) \le h(x_2-x_1);$$
hence
\begin{equation}
\label{eq:h_basic}
|h(x_2)-h(x_1)| \le h(x_2-x_1).
\end{equation}

The following simple lemma relates 
$\binom Nm$ to the binary entropy.
\begin{lemma} %Lemma A
\label{lem:entropyBinomial}
For integers $N \ge m \ge 0$, we have 
$$
\binom{N}{m} \le
\left(\frac{N}{m}\right)^m \left(\frac{N}{N-m}\right)^{N-m} = 2^{N h(m/N)}.
$$
\end{lemma}
\begin{proof}
Set $p=\xfrac{m}{N}$. 
If $X$ has
the binomial distribution $\mathrm{Bin}(N,p)$,
then 
$$1 \ge \Pr[X=m] = \binom{N}{m} p^m (1-p)^{N-m}
$$
and thus 
\begin{equation*}
\binom Nm \le p^{-m} (1-p)^{-(N-m)}
=
\left(\frac{N}{m}\right)^m \left(\frac{N}{N-m}\right)^{N-m}
=2^{Nh(p)}.
\qedhere
\end{equation*}
\end{proof}

We will need the following simple lemma about  hereditary classes of graphs
\cite{SJ255}:
\begin{lemma} %Lemma 0
\label{lem:hered}
Let $\QQ$ be a hereditary class of graphs and let $W$ be a graphon.
Then $W \in \cqoo$ if and only if $p(F;W)=0$ when $F \not\in \QQ$.
\end{lemma}
\begin{proof}
If $F \not\in \QQ$, then $p(F;G)=0$ for every $G \in \QQ$
since $\QQ$ is hereditary, and thus $p(F;W)=0$ for every $W \in \cqq$ by
continuity.

For the converse, assume that $p(F;W)=0$ when $F \not\in \QQ$.
Thus $p(F;W)>0\implies F\in\QQ$.
By \eqref{gnw}, if $\P(G(n,W)=H)>0$, then $p(H;W)>0$ and thus $H\in\QQ$.
Hence, $G(n,W)\in\QQ$ almost surely.
The claim follows from 
the fact~\cite{BCLSV1} that almost surely $G(n,W)$ converges to $W$ as \ntoo.
\end{proof}

Next we recall that the \emph{cut norm} of an $n \times n$ matrix
$A=(A_{ij})$  is
defined by  
$$\|A\|_\square := \frac{1}{n^2} \max_{S,T \subseteq [n]} \left|\sum_{i \in S, j \in T} A_{ij} \right|.$$ 
Similarly, the \emph{cut norm} of a measurable $W:[0,1]^2 \to \mathbb{R}$ is defined as 
$$\|W\|_\square = \sup \left| \iint f(x) W(x,y) g(y) \dd x \dd y\right|,$$
where the supremum is over all measurable functions $f,g:[0,1] \to
\{0,1\}$. 
(See \cite{BCLSV1} and \cite{SJ249} for other versions, equivalent within
constant factors.)
We use also the notation, for two graphons $W_1$ and $W_2$, 
\begin{equation}\label{d}
d_\square(W_1,W_2)\=  \| W_1 - W_2\|_\square.
\end{equation}
The \emph{cut distance} between two graphons  
$W_1$ and $W_2$ is defined as
\begin{equation}\label{dcut}
\delta_\square(W_1,W_2)\= \inf_{W_2'} d_\square(W_1,W'_2)
=
\inf_{W_2'} \| W_1 - W'_2\|_\square,  
\end{equation}
where the infimum is over all graphons $W'_2$ that are 
equivalent to $W_2$.
(See \cite{BCLSV1} and \cite{SJ249} for other, equivalent, definitions.)
The cut distance is a pseudometric on $\cW_0$, with $\gd_\square(W_1,W_2)=0$
if and only if $W_1$ and $W_2$ are equivalent.

The cut distance between two graphs $F$ and $G$ is defined as
$\delta_\square(F,G)=\delta_\square(W_F,W_G)$.  
We similarly write $\delta_\square(F,W)=\delta_\square(W_F,W)$, 
$d_\square(F,W)=d_\square(W_F,W)$ and so on.

The cut distance is a central notion in the theory of graph limits. For
example it is known (see~\cite{BCLSV1} and \cite{LovaszBook})  
that a graph sequence $(G_n)$ with $\vv{G_n} \to \infty$ converges to a graphon $W$ if and only if the sequence $(W_{G_n})$ 
converges to $W$ in cut distance. Similarly, convergence of a sequence of
graphons  in $\mathcal{W}_0$ is the same as convergence in cut distance; hence,
the cut distance induces a metric on $\cuoo$ that defines its
topology. 

Let $\PP$ be a partition of the interval $[0,1]$ into $k$ measurable sets
$I_1,\ldots,I_k$.  Then $I_1,\ldots,I_k$ divide the unit square  
$[0,1]^2$ into $k^2$ measurable sets $I_i \times I_j$. We denote the
corresponding $\sigma$-algebra by $\BB_\PP$; note that if $W$ is a graphon,
then $\E[W\mid\BB_\PP]$ is the graphon that is constant on each set
$I_i\times I_j$ and obtained by averaging $W$ over each such set.
A partition of the interval $[0,1]$ 
into $k$ sets is called an \emph{equipartition} if all sets are of measure
$1/k$. 
We let $\cPP_k$ denote the equipartition of $\oi$ into $k$ intervals of
length $1/k$, and 
write, for any graphon $W$,
\begin{equation}  \label{bW}
\bW_k\=\E[W\mid\BB_{\cPP_k}].
\end{equation}

Similarly, if  $\PP$ is  a partition of $[n]$ into sets $V_1,\ldots,V_k$, then we consider the 
corresponding partition $I_1,\ldots,I_k$ of $[0,1]$ (that is $x \in (0,1]$ belongs to $I_j$ if and only if $\lceil x_j \rceil \in V_j$) and again 
we denote the corresponding $\sigma$-algebra on $[0,1]^2$ by $\BB_\PP$. A
partition  $\PP$  of $[n]$ into $k$ sets is called an \emph{equipartition}  
if each part is of size $\lfloor n/k \rfloor$ or $\lceil n/k \rceil$.

The 
graphon version of the
weak regularity lemma proved by Frieze and Kannan~\cite{MR1723039}, 
see also \cite{LovaszSzegedy:analyst} and 
\cite[Sections 9.1.2 and  9.2.2]{LovaszBook}, %Lemma 9.3, Exercise 9.7 
says that
for every every graphon $W$ and every $k \ge 1$,
there is an
equipartition $\PP$ of $[0,1]$ into $k$ sets such that 
\begin{equation}
 \label{weakRegGraphon}
\left\| W - \Ex[W\mid\BB_\PP] \right\|_\square \le \frac{4}{\sqrt{\logb k}}.
\end{equation}

Let us close this section with the following simple lemma.
 Part (ii) has been proved by
Chatterjee and Varadhan \cite{ChatterjeeVaradhan}, but
we include a (different) proof for completeness.

\begin{lemma}
\label{lem:EntSemicont}
The function $\Ent(\cdot)$ satisfies the following properties:
\begin{romenumerate}
\item 
If\/ $W$ is a graphon and $\mathcal{P}$ is a measurable partition of
$[0,1]$, then 
$$\Ent(\Ex[W\mid\BB_\PP]) \ge \Ent(W).$$
\item 
The function $\Ent(\cdot)$ is lower semicontinuous on $\cW_0$
(and, equivalently, on $\cuoo$). 
I.e., if\/ $W_m\to W$ in $\mathcal{W}_0$ as \mtoo, then
$$
\limsup_{m \to \infty} \Ent(W_m) \le \Ent(W).
$$
\end{romenumerate}
\end{lemma}

\begin{remark}
$\Ent(\cdot)$   is \emph{not} continuous. For example, let $G_n$ be a
  quasirandom sequence of graphs with $G_n\to W=\frac12$ (a constant
  graphon);
then $W_{G_n}\to W=\frac 12$ in $\cW_0$, but $\Ent(W_{G_n})=0$ and $\Ent(W)=1$.
\end{remark}

\begin{proof}
Part~(i) follows from Jensen's inequality and  concavity of $h$. 

To prove (ii),
note that we can assume $\|W_m - W \|_\square \to 0$. For every $k \ge 1$,
let $\cPP_k$ be the partition of $[0,1]$  
into $k$  consecutive  intervals of equal measure $1/k$. 
Consider the step graphons
$\Ex[W_{m}\mid \BB_{\cPP_k}]$ and $\Ex[W\mid \BB_{\cPP_k}]$. For each
$k$, $\Ex[W_{m}\mid \BB_{\cPP_k}]$ converges to $\Ex[W\mid \BB_{\cPP_k}]$ almost
everywhere as $\mtoo$, and
thus by \eqref{entw} and dominated convergence,
$$
\lim_{m \to \infty} \Ent(\Ex[W_{m}\mid \BB_{\cPP_k}]) = \Ent(\Ex[W\mid \BB_{\cPP_k}]).$$
Consequently, using (i),
$$\limsup_{m \to \infty} \Ent(W_m) \le \limsup_{m \to \infty}
\Ent(\Ex[W_{m}\mid \BB_{\cPP_k}])=\Ent(\Ex[W\mid \BB_{\cPP_k}]).
$$
Finally, let $k \to \infty$. Then $\Ex[W\mid \BB_{\cPP_k}] \to W$ almost
everywhere, and 
thus $\Ent(\Ex[W\mid \BB_{\cPP_k}])\to\Ent(W)$.
\end{proof}

\section{Number of graphs and Szem\'eredi partitions}
In this section we prove some of the key lemmas  needed in this paper. These lemmas 
provide various estimates on the number of graphs on $n$ vertices that are
close to a graphon in cut distance. For an integer $n \ge 1$, a parameter
$\delta>0$, 
and a graphon $W$, define 
\begin{align}\label{hN}
\widehat{N}_\square(n,\delta;W) 
&:= \left| \{G \in \LL_n: d_\square( W_G,W) \le \delta \} \right|    
\intertext{and}
N_\square(n,\delta;W) 
&:= \left| \left\{ G \in \LL_n: \delta_\square(G,W) \le \delta \right\} \right|.
\end{align}

Since 
$\delta_\square(G,W) \le d_\square( W_G,W)$, \cf{} \eqref{dcut},
we have trivially
\begin{equation}\label{hNN}
\widehat{N}_\square(n,\delta;W) \le N_\square(n,\delta;W).   
\end{equation}
We will show an estimate in the opposite direction, showing that for our
purposes, 
$\widehat{N}_\square(n,\delta;W)$ and $N_\square(n,\delta;W)$ are not too
different. We begin with the following estimate.
We recall that $\bW_k\=\E[W\mid\BB_{\cPP_k}]$ is obtained by averaging $W$ over
squares of side $1/k$, see \eqref{bW}.

\begin{lemma} 
\label{lem:C}
Let $W$ be a graphon. If\/ $G\in\LL_n$, then there is a
graph $\tG\in\LL_n$  isomorphic to $G$ such that
\begin{equation}\label{lemc}
d_\square(\tG,W) \le 
\delta_\square(G,W) + 2 d_\square (W,\bW_n)
+\frac{18}{\sqrt{\logb n}}.  
\end{equation}
\end{lemma}
\begin{proof}
Regard $\bW_n$ as a weighted graph on $n$ vertices, and 
consider the random graph $G(\bW_n)$ on $[n]$, defined by connecting
each pair 
$\set{i,j}$ of nodes by an edge $ij$ with probability
$\bW_n(i/n,j/n)$,
%=n^2\int_{\frac{i-1}{n}}^{\frac{i}{n}}
%\int_{\frac{j-1}{n}}^{\frac{j}{n}} W(x,y) \dd x \dd y$,
independently for different pairs. 
By \cite[Lemma 10.11]{LovaszBook}, with positive
probability (actually at least $1-e^{-n}$),
$$
d_\square(G(\bW_n), \bW_n) \le \frac{10}{\sqrt{n}}.$$
Let $G'$ be one realization of $G(\bW_n)$ with
\begin{equation}
\label{eq:randomClose}
d_\square(G', \bW_n) \le \frac{10}{\sqrt{n}}.
\end{equation}
Then,
by the triangle inequality and \eqref{eq:randomClose},
\begin{equation}\label{c2}
  \begin{split}
\delta_\square(G,G') 
&\le  \delta_\square(G, W) +
\gdcut(W,\bW_n)+\gdcut(\bW_n,G')
\\&
\le 
\delta_\square(G,W)+\dcut(W,\bW_n)+\frac{10}{\sqrt{n}}.
  \end{split}
\end{equation}
Since $G$ and $G'$ both are graphs on $[n]$, we can 
by \cite[Theorem 9.29]{LovaszBook}
permute the labels of
$G$ and obtain a graph $\tG\in\LL_n$ such that
\begin{equation}\label{c3}
\dcut(\tG,G') \le
\delta_\square(G,G')+\frac{17}{\sqrt{\logb n}}.
\end{equation}
Consequently, 
by the triangle inequality again
and \eqref{eq:randomClose}--\eqref{c3},
\begin{equation*}
  \begin{split}
\dcut(\tG,W)
&\le \dcut(\tG,G')+\dcut(G',\bW_n)+\dcut(\bW_n,W)  
\\&
\le 
\delta_\square(G,G')+\frac{17}{\sqrt{\logb n}}
+\frac{10}{\sqrt{ n}}
+\dcut(W,\bW_n)
\\&
\le 
\delta_\square(G,W)
+2d_\square(W,\bW_n)
+\frac{17}{\sqrt{\logb n}}
+\frac{20}{\sqrt{n}}.
  \end{split}
\end{equation*}
The claim follows for $n>2^{20}$, say; for smaller $n$ it is trivial since
$\dcut(\tG,W)\le1$ for every $\tG$.
\end{proof}

\begin{lemma}\label{L:Cx} 
  For any graphon $W$, $\gd>0$ and $n\ge1$,
  \begin{equation}
N_\square(n,\delta;W)
\le n!
\widehat{N}_\square(n,\delta+\eps_n;W),
  \end{equation}
where $\eps_n\=18/\sqrt{\logb n}+2\dcut(\bW_n,W)\to0$ as \ntoo.
\end{lemma}

\begin{proof}
  By \refL{lem:C}, if $G\in\LL_n$ and $\gdcut(G,W)\le\gd$, then
$\dcut(\tG,W)\le\gd+\eps_n$ for some relabelling $\tG$ of $G$. There are at
  most $\widehat{N}_\square(n,\delta+\eps_n;W)$ such graphs $\tG$ by \eqref{hN},
and each corresponds to at most $n!$ graphs $G$.
Finally, note the well-known fact that $\dcut(\bW_n,W)\le
\norm{\bW_n-W}_{L^1}\to0$ as \ntoo. 
%by Lebesgue's differentiation theorem (or by simple martingale theory).
\end{proof}

\begin{remark}
  The bound \eqref{lemc} in \refL{lem:C} is not valid without the term
  $\dcut(W,\bW_n)$. For a simple example, let $n$ be even and let
$G$ be a balanced complete  bipartite graph. Further, let 
$W\=W_{G}(\frax{nx},\frax{ny})$, where $\frax{x}$ denotes the fractional
  part. (Thus $W$ is 
obtained by partitioning $\oi^2$ into $n^2$ squares and putting a
  copy of $W_{G}$ in each. Furthermore, 
$W=W_{G'}$ for a blow-up $G'$ of $G$ with $n^2$ vertices.)
Then $W$ is equivalent to $W_{G}$, so $\gdcut(G,W)=0$. Furthermore,
$\bW_n=1/2$ (the edge density), and it is easily seen that 
for any relabelling $\tG$ of $G$,
$\dcut(\tG,W)\ge\dcut(\tG,\bW_n)\ge \frac{1}8$. Hence the left-hand side of
\eqref{lemc} does not tend to 0 as \ntoo; thus the term $\dcut(W,\bW_n)$ 
is needed. 
\end{remark}

After these preliminaries, we turn to estimating 
$\widehat{N}_\square(n,\delta;W)$ and ${N}_\square(n,\delta;W)$ 
using $\Ent(W)$. 

\begin{lemma}
\label{lem:ball_w}
For every graphon $W$  and for every $\delta>0$,
$$\liminf_{n \to \infty} \frac{\logb \widehat{N}_\square(n,\delta;W)}{\binom{n}{2}}  \ge
\Ent(W).$$
\end{lemma}
\begin{proof}
Consider the random graph $G(n,W)\in\LL_n$. 
As shown in \cite{MR2274085}, $G(n,W)\to W$ almost surely, and thus in
probability; in other words,
the probabilities $p_n := \Pr[\delta_{\square}(G(n,W),W) \le \delta]$
converge to $1$ as \ntoo.
Moreover it is shown in~\cite{Aldous} and \cite[Appendix D]{SJ249} that 
\begin{equation}
  \label{sjw}
\lim_{n \to \infty }\frac{\HH(G(n,W))}{\binom{n}{2}}  = \Ent(W), 
\end{equation} 
where $\HH(\cdot)$ denotes the usual entropy of a (discrete) random variable.

Let $I_n := \etta_{[\delta_\square(G(n,W),W) \le \delta]}$ so
that
$\Ex[I_n]=p_n$. We have, by simple standard results on entropy,
\begin{equation*}
  \begin{split}
 \HH(G&(n,W)) = \Ex [\HH(G(n,W)\mid I_n)] + \HH(I_n) \\
&= p_n \HH(G(n,W)\mid I_n=1) + (1-p_n ) \HH(G(n,W)\mid I_n=0) +h(p_n) \\
& \le  p_n \logb N_\square(n,\delta; W) + (1-p_n) \binom{n}{2} + h(p_n) \\
& \le  \logb N_\square(n,\delta;W) + (1-p_n) \binom{n}{2} + 1
\\&
=
  \logb N_\square(n,\delta;W) + o\bigpar{n^2}.
  \end{split}
\end{equation*}
By \refL{L:Cx}, this yields 
\begin{equation*}
  \begin{split}
 \HH(G&(n,W))
\le
  \logb \hN_\square(n,\delta+\eps_n;W) + o\bigpar{n^2}
  \end{split}
\end{equation*}
for some sequence $\eps_n\to0$.
The result follows now from \eqref{sjw}, 
if we replace $\gd$ by $\gd/2$.
\end{proof}

We define, for convenience, for $x\ge0$,
\begin{equation}
  \hx(x)\=h\bigpar{\min(x,\tfrac12)};
\end{equation}
thus $\hx(x)=h(x)$ for $0\le x\le \frac12$, and $\hx(x)=1$ for $x>\frac12$.
Note that $\hx$ is non-decreasing.

\begin{lemma}
\label{lem:boundHatN}
Let $W$ be a graphon, $n \ge k \ge 1$ be integers and  $\delta>0$.
For any equipartition $\PP$ of $[n]$  
into $k$ sets, we have 
$$
\frac{\logb \widehat{N}_\square(n,\delta;W)}{n^2}  \le \frac{1}{2}
\Ent(\Ex[W\mid\BB_\PP]) + \frac{1}{2} \hx(4k^2 \delta) + 2 k^2 \frac{\logb
  n}{n^2}.
$$ 
\end{lemma}
\begin{proof}
Denote the sets in $\PP$ by  $V_1,\ldots,V_k \subseteq [n]$ and their sizes
by $n_1,\ldots,n_k$, and let $I_1,\ldots,I_k$ be the subsets in  
the corresponding partition of $[0,1]$. 

Let $w_{ij}$ denote the value of
$\Ex[W\mid \BB_\PP]$ on $I_i \times I_j$. Suppose that $G\in\LL_n$ 
and $\|W_G -
W\|_\square \le \delta$. Let $e(V_i,V_j)$ be the number of edges in $G$
from $V_i$ to
$V_j$ when $i\neq j$, and twice the number of edges with both endpoints in
$V_i$ when $i=j$. Then
$$e(V_i,V_j) = n^2 \int_{I_i \times I_j} W_G(x,y) \dd x \dd y,$$
and thus
$$\bigabs{ e(V_i,V_j) - w_{ij} n_i n_j} 
= n^2 \left| \int_{I_i \times I_j}
(W_G(x,y)-W(x,y)) \dd x \dd y \right| \le \delta n^2.$$
Hence
\begin{equation}
\label{eq:edgeDeviation}
\left|\frac{e(V_i,V_j)}{n_in_j} -w_{ij} \right| \le \frac{\delta n^2}{n_in_j}
\le \delta \left(\frac{n}{\lfloor n/k \rfloor} \right)^2 \le 4 k^2 \delta.
\end{equation}
Fix numbers $e(V_i,V_j)$ satisfying (\ref{eq:edgeDeviation}), and let $N_1$ be
the number of graphs on $[n]$ with these $e(V_i,V_j)$.
By Lemma~\ref{lem:entropyBinomial}, for $i \neq j$, the edges in $G$ between
$V_i$ and $V_j$ can be chosen in
\begin{equation}
\label{eq:graphs_num_offdiag}
\binom{n_in_j}{e(V_i,V_j)} \le 2^{n_in_j h(e(V_i,V_j)/n_in_j)}
\end{equation}
number of ways. For $i=j$, the edges in $V_i$
may be chosen in
\begin{equation}
\label{eq:graphs_num_diag}
\binom{\binom{n_i}{2}}{\frac{1}{2} e(V_i,V_i)} 
\le 2^{\binom{n_i}{2} 
h(\frac{1}{2}e(V_i,V_i)/\binom{n_i}{2})} 
\le
2^{\frac{1}{2}n_i^2 h(e(V_i,V_i)/n_i^{2})}
\end{equation}
number of ways, where the second inequality holds because $h$ is concave 
with $h(0)=0$ and thus $h(x)/x$ is decreasing.

Consequently, by (\ref{eq:graphs_num_offdiag}) and
(\ref{eq:graphs_num_diag}),
\begin{equation*}
  \begin{split}
\logb N_1 
&\le \sum_{i < j} n_in_j h\bigpar{e(V_i,V_j)/n_in_j} + \frac{1}{2} \sum_{i}
n_i^2 h\bigpar{e(V_i,V_i)/n_i^2}
\\&
= \frac{1}{2} \sum_{i,j=1}^k n_in_j h\bigpar{e(V_i,V_j)/n_in_j}.	
  \end{split}
\end{equation*}
Using (\ref{eq:edgeDeviation}) and (\ref{eq:h_basic}), we obtain
$$
\logb N_1  \le \frac{1}{2} \sum_{i,j=1}^k n_in_j 
\bigpar{h(w_{ij})+\hx(4k^2 \delta)},
$$
and thus
\begin{equation*}
  \begin{split}
n^{-2} \logb N_1 
&\le \frac{1}{2} \sum_{i,j} |I_i||I_j| \bigpar{h(w_{ij})+\hx(4k^2 \delta)}
\\&
= \frac{1}{2} \Ent(\Ex[W\mid\BB_\PP]) + \frac{1}{2} \hx(4k^2 \delta).	
  \end{split}
\end{equation*}

Each $e(V_i,V_j)$ may be chosen in at most $n^2$ ways, and thus the total number
of choices is at most $n^{2k^2}$, and we obtain
$\widehat{N}_\square(n,\delta;W) \le n^{2k^2} \max N_1.$ 
Consequently, 
\begin{equation*}
n^{-2} \logb \widehat{N}_\square(n,\delta;W) \le \frac{1}{2}
\Ent(\Ex[W\mid\BB_\PP]) + \frac{1}{2} \hx(4k^2 \delta) 
+ 2k^2 \frac{\logb n}{n^2}.
\qedhere  
\end{equation*}
\end{proof}

\begin{lemma}\label{lem:bound_N}
Let $W$ be a graphon. Then for any $k \ge 1$, $\gd>0$ and any equipartition
$\mathcal{P}$ of $[0,1]$ into $k$ sets,
$$\limsup_{n \to \infty} \frac{\logb N_\square(n,\delta;W)}{\binom{n}{2}} \le
\Ent(\Ex[W\mid\BB_\PP]) + \hx(4 k^2 \delta).$$
Consequently
$$\lim_{\delta \to 0} \limsup_{n \to \infty} 
\frac{\logb N_\square(n,\delta;W)}{\binom{n}{2}} 
\le \Ent(\Ex[W\mid\BB_\mathcal{P}]).$$
\end{lemma}
\begin{proof}
By a suitable measure preserving re-arrangement $\sigma:[0,1] \to[0,1]$,
we may assume that $\cP$ is the partition $\cPP_k$ 
into $k$  intervals $((j-1)/k,j/k]$ of length $1/k$. 

For every $n>1$, let $\PP_n$ be the corresponding
equipartition of $[n]$ into $k$ sets $P_{n1},\dots,P_{nk}$ 
where 
$P_{nj}\=\set{i:\floor{(j-1)n/k}<i\le \floor{jn/k}}$.
Note that $\Ex[W\mid \BB_{\PP_n}]$ converges
to $\Ex[W\mid\BB_\PP]$ almost everywhere as \ntoo, and hence
$$
\lim_{n \to \infty} \Ent(\Ex[W\mid \BB_{\PP_n}]) 
= \Ent(\Ex[W\mid \BB_{\PP}]).
$$
Then by Lemmas \refand{L:Cx}{lem:boundHatN}, we have, with $\eps_n\to0$,
\begin{multline*}
 \frac{\logb N_\square(n,\delta;W)}{n^2} 
\le
\frac{\logb (n!)}{n^2}+\frac12 \Ent(\Ex[W\mid \BB_{\PP_{n}}]) 
\\
+ \frac12 \hx\left(4k^2 \left(\delta + \eps_n\right)\right) 
+ 2 k^2 \frac{\logb n}{n^2}   
\end{multline*}
and the result follows by letting \ntoo.
\end{proof}

We can now show our main lemma.
As usual, if $A$ is a set of graph limits, we define
$ \delta_\square(G,A) \=\inf_{W \in A}\delta_\square(G,W) $.

\begin{lemma}
\label{lem:T1}
Let $A\subseteq\cuoo$ be a closed set of graph limits and let
$$
N_\square(n,\delta;A) := 
\left| \left\{G \in \mathcal{L}_n: 
\delta_\square(G,A) \le \delta  \right\} \right|.$$
Then
\begin{equation}
\label{eq:thm_T1}
\lim_{\delta \to 0} 
\liminf_{n \to \infty} \frac{\logb N_\square(n,\delta;A) }{\binom{n}{2}}
= \lim_{\delta \to 0} \limsup_{n \to \infty} 
\frac{\logb N_\square(n,\delta;A)}{\binom{n}{2}}
= \max_{W \in A} \Ent(W).
\end{equation}
\end{lemma}
\begin{proof}
First note that the maximum in the right-hand side of (\ref{eq:thm_T1}) exists
as a consequence of the semicontinuity of $\Ent(\cdot)$ in
Lemma~\ref{lem:EntSemicont}~(ii) and the compactness of $A$.

Let $\delta>0$ and $k \ge 1$. Since $A$ is a compact subset of $\cuoo$, 
there exists a finite set of graphons $\{W_1,\ldots,W_m\} \subseteq A$ such
that  $\min_{i} \delta_\square(W,W_i) \le \delta$ for each $W \in A$. Hence
\begin{equation}
  \label{ba}
N_\square(n,\delta;A) \leq \sum_{i=1}^m N_\square(n,2\delta;W_i).
\end{equation}
By (\ref{weakRegGraphon}), for each $W_i$, we can choose an
equipartition 
$\PP_i$ of $[0,1]$ into at most 
$k$ sets such that 
\begin{equation}\label{bc}
\|W_i-\Ex[W_{i} |\BB_{\PP_i}]\|_\square \le \frac{4}{\sqrt{\logb k}}.  
\end{equation}
By \eqref{ba} and Lemma~\ref{lem:bound_N},
\begin{equation}\label{bb}
  \begin{split}
\limsup_{n \to \infty} \frac{\logb N_\square(n,\delta;A)}{\binom{n}{2}}  
&\le
\max_{i \le m}
\limsup_{n \to \infty} \frac{\logb N_\square(n,2 \delta; W_i)}{\binom{n}{2}} 
\\&
\le \max_{i \le m} \Ent(\Ex[W_{i}\mid \BB_{\PP_i}])+\hx(8k^2 \delta).	
  \end{split}
\end{equation}
For each $k \ge 1$, take $\delta=2^{-k}$ and let $i(k)$ denote the index
maximizing $\Ent(\Ex[W_{i}| \BB_{\PP_i}])$ in \eqref{bb};
further let $W'_k\=W_{i(k)}$ and $W_k''\=\Ex[W_{i(k)}| \BB_{\PP_{i(k)}}]$.
Thus $W'_k \in A$, and by \eqref{bc}--\eqref{bb},
\begin{equation}
\label{eq:1_thm_T1}
\|W'_k-W''_k\|_\square \le \frac{4}{\sqrt{ \logb k}}
\end{equation}
and
\begin{equation}
\label{eq:2_thm_T1}
\limsup_{n \to \infty} \frac{\logb N_\square(n,2^{-k};A)}{\binom{n}{2}}  
\le \Ent(W_k'')
+\hx(8 k^2 2^{-k}).
\end{equation}
Since $A$ is compact, we can select a subsequence such that $W'_k$ converges,
and then $W'_k \to W'$ for some $W' \in A$. By (\ref{eq:1_thm_T1}),
also $W''_k \to W'$ in $\mathcal{W}_0$ and thus Lemma~\ref{lem:EntSemicont}
shows 
that
\begin{equation}
  \label{jane}
\limsup_{k \to \infty} \Ent(W''_k) \le \Ent(W').
\end{equation}
Since $N_\square(n,\delta;A)$ is an increasing function of $\delta$, letting  $k
\to \infty$, it follows from (\ref{eq:2_thm_T1}) and \eqref{jane} that
\begin{equation*}
  \begin{split}
\lim_{\delta \to 0} \limsup_{n \to \infty} 
\frac{ \logb N_\square(n,\delta;A)}{\binom{n}{2}}  
&=
\lim_{k \to \infty} \limsup_{n \to \infty} 
\frac{ \logb N_\square(n,2^{-k};A)}{\binom{n}{2}}  
\\&
\le \Ent(W') \le \max_{W \in A} \Ent(W),
  \end{split}
\end{equation*}
which shows that the \rhs{} in \eqref{eq:thm_T1} is an  upper bound.

To see that the \rhs{} in \eqref{eq:thm_T1} also is a lower bound,
note that \eqref{hNN} implies that for every $W\in A$,
\[
 N_\square(n,\gd;A)
\ge  N_\square(n,\gd;W)
\ge 
\widehat{N}_\square(n,\delta;W).
\]
The sought lower bound thus follows from 
Lemma~\ref{lem:ball_w},
which completes the proof.
\end{proof}

\section{Proofs of Theorems~\ref{thm:T2}--\ref{thm:TC}\label{sec:proofT2}}

\begin{proof}[Proof of Theorem~\ref{thm:T2}]
Let $\delta >0$. First observe that for sufficiently large $n$, if 
$G \in \QQ_n$, then $\delta_\square(G,\cqoo) < \delta$. Indeed, if not, then 
we could find a sequence $G_n$ with $\vv{G_n} \to \infty$ and
$\delta_\square(G,\cqoo) \ge \delta$. Then, by compactness,  $G_n$ would have a 
convergent subsequence, but the limit cannot be in $\cqoo$ which is a
contradiction. Consequently for  
sufficently large $n$, we have 
$|\QQ_n|\le |\ql_n|\le N_\square(n,\delta; \cqoo)$. Thus
\begin{equation*}
 \limsup_{n \to \infty} \frac{ \logb |\QQ_n|}{\binom{n}{2}}
\le
\limsup_{n \to \infty} \frac{\logb N_\square(n,\delta; \cqoo)}
 {\binom{n}{2}} .
\end{equation*}
The result now follows from the Lemma~\ref{lem:T1}.
\end{proof}

\begin{proof}[Proof of \refT{thm:T4}] 
Let $W$ be a graphon representing some $\gG \in \cqoo$ and consider
the random graph $G(n,W)\in\LL_n$. Since $\QQ$ 
is hereditary, it is easy to see that almost surely $G(n,W) \in \ql_n$,
see \refL{lem:hered} and \eqref{gnw} or \cite{SJ255}. Consequently,
letting $\HH(G(n,W))$ denote the entropy of the random graph $G(n,W)$ (as a
random variable in the finite set $\ql_n$),
$$\HH(G(n,W)) \le \logb |\ql_n|.$$
Hence, \eqref{sjw} and \eqref{n!} show that,
 for every $W \in \cqoo$,
$$
\liminf_{n \to \infty} \frac{ \logb |\QQ_n|}{\binom{n}{2}} 
=\liminf_{n \to \infty} \frac{ \logb |\ql_n|}{\binom{n}{2}} 
\ge \Ent(W). $$
The result now follows from Theorem~\ref{thm:T2}.
\end{proof}

\begin{proof}[Proof of \refT{thm:T3}]
(i). Let $\delta>0$ and let $B(\delta)=\{G: \delta_\square(G,\W_\QQ) <
  \delta\}$. The 
conclusion means that, for any $\gd$, 
$\Pr[G_n \in B(\delta)] \to 1$ as $n \to \infty$, \ie{}
$$\frac{|\QQ_n \cap B(\delta)|}{|\QQ_n|} \to 1.$$
If this is not true, then for some $c>0$ there are infinitely many $n$ with 
\begin{equation}
\label{eq:2_thm_T3}
|\QQ_n \setminus B(\delta)| \ge c |\QQ_n|.
\end{equation}
Consider the graph property
$\QQ^* := \QQ \setminus B(\delta)$. By (\ref{eq:2_thm_T3}) and the
assumption (\ref{eq:thm_T3})  
$$\limsup_{n \to \infty} \frac{\logb |\QQ_n^*|}{\binom{n}{2}} =\Ent(\W_\QQ).$$
Hence Theorem~\ref{thm:T2} shows that 
$\Ent(\W_\QQ) \le \max_{\W \in \widehat{\QQ^*}}
\Ent(\W)$.
So there exists $\W^* \in \widehat{\QQ^*}$ such that $\Ent(\W_\QQ) \le \Ent(\W^*)$.

On  the other hand,
$\overline{\QQ^*} \subseteq \overline{\QQ}$ and 
$\W_\QQ\not\in\overline{\QQ^*}$ so
$\widehat{\QQ^*}=\overline{\QQ^*}\cap\cuoo \subseteq
\widehat{\QQ}\setminus\set{\W_\QQ}$, 
but by assumption $\Ent(\W) < \Ent(\W_\QQ)$ for 
$\W \in \widehat{\QQ} \setminus \{\W_\QQ\}$. This yields a contradiction which
completes the proof of (i).

(ii). 
The labelled case follows in the same way, now using $\ql_n$ and
\eqref{n!}. 
\end{proof}

\begin{proof}[Proof of \refT{thm:TC}]
Let $\delta>0$ and let $\UB= \{G: \delta_\square(G, \supp(\mu))<\delta \}$. Then
$\UB$ is an open neighborhood of $\supp(\mu)$ in $\cuq$
and thus the assumption that $G_n$ converges in distribution to  $\mu$
implies that $\lim_{n \to \infty} \Pr[G_n \in \UB]=1$. We have,
similarly to the proof of \refL{lem:ball_w},
\begin{equation*}
  \begin{split}
\HH(G_n) &= \E[\HH(G_n\mid \etta_{[G_n \in \UB]})] + \HH(\etta_{[G_n \in \UB]}) 
\\ 
&= \Pr[G_n \in \UB] \HH(G_n\mid G_n \in \UB) + \Pr[G_n \not\in \UB] \HH(G_n| G_n
\not\in \UB) \\
&\hskip2em + \hx(\Pr[G_n \in \UB]) \\
&\le  \HH(G_n\mid G_n \in \UB) + \Pr[G_n \not\in \UB] \binom{n}{2}+1 \\
&\le
 \logb N_\square(n,\delta;\supp(\mu)) + \Pr[G_n \notin \UB] \binom{n}{2} +1.	
  \end{split}
\end{equation*}
Hence using $\lim_{n \to \infty} \Pr[G_n \not\in \UB]=0$,
$$ 
\limsup_{n \to \infty} \frac{\Ent(G_n)}{\binom{n}{2}} \le 
\limsup_\ntoo
\frac{\logb N_\square(n,\delta,\supp(\mu))}{\binom{n}{2}}.$$
The result follows from Lemma~\ref{lem:T1} by letting $\delta \to 0$.
\end{proof}

\section{Proof of Theorem~\ref{thm:TR}}

The stability version of Tur\'an's theorem, due to 
Erd{\H o}s and Simonovits \cite{MR0227049,MR0233735}, 
is equivalent to the following statement for graphons, see 
\cite[Lemma 23]{Pikhurko10} for a detailed proof and further explanations of
the connection. 
\begin{theorem}[\cite{Pikhurko10}]  %[\cite{MR0227049,MR0233735}]
\label{thm:ErdosSimonovits}
If a graphon $W$ is $K_{r+1}$-free, then $\iint W \le 1 -\frac{1}{r}$ with
equality if and only if $W$ is equivalent to the graphon $W_{K_r}$.
\end{theorem}

Recall the definition of randomness support and randomness support gra\-phon,
see \eqref{rand}.

\begin{lemma}
\label{lem:withclique}
Let $1\le r<\infty$.
If $\QQ$ is a hereditary graph class and $W \in \widehat{\QQ}$ has a
randomness support 
graphon that is not $K_r$-free, then 
there exists $s\in\set{0,1,\dots,r}$ such that $\wwrs\in\cqoo$.
In particular,
$\widehat{\QQ} \cap R_r \neq \emptyset$.
\end{lemma}
\begin{proof}
First define $W'(x,y) := W(x,y)$ if $W(x,y) \in \{0,1\}$ and
$W'(x,y):=\frac{1}{2}$ if $0<W(x,y)<1$. Then 
$W'(x,y)\in\set{0,\frac12,1}$ for all $(x,y)$. 
Moreover,
$W'$ has the same randomness
support 
as $W$ and it is easily seen that for any graph $F$, $p(F;W)=0$ if and only if
$p(F;W')=0$. It  follows from Lemma~\ref{lem:hered} that  $W' \in
\widehat{\QQ}$. 

Let $x_1,\ldots,x_r \in (0,1)$ be chosen at random, uniformly and independently.
By assumption, with positive probability,  we have $W'(x_i,x_j)=1/2$ for all
pairs $i \neq j$. 
 Choose one such sequence $x_1,\ldots,x_r$ such that furthermore
 $x_1,\ldots,x_r$ 
are distinct and $(x_i,x_j)$ is a Lebesgue point of $W'$ when $i \neq j$;
this is possible since the additional conditions hold almost surely. 
Let $m$ be a positive integer, and set $J_{i,\epsilon}=(x_i-\epsilon,x_i+\epsilon)$ for
$\epsilon>0$ and $i=1,\ldots,r$.
If $\epsilon$ is sufficiently small, then these intervals are disjoint
subintervals of $(0,1)$, 
and further, if $i \neq j$, then
\begin{equation}
  \label{b80}
\lambda\left(\left\{(x,y) \in J_{i,\epsilon} \times J_{j,\epsilon} :
W'(x,y)=\tfrac{1}{2} \right\}\right) 
> \lrpar{1-\tfrac1m} |J_{i,\epsilon}| \cdot |J_{j,\epsilon}|, 
\end{equation}
where $\lambda$ denotes the Lebesgue measure.

Take such an $\epsilon$ and let $W'_m$ be the graphon obtained by scaling the restriction of $W'$
to $(\bigcup_{i=1}^r J_{i,\epsilon}) \times (\bigcup_{i=1}^r
J_{i,\epsilon})$   to a graphon in the natural way, by mapping
$I_i:=(\frac{i-1}{r},\frac ir]$  linearly to $J_{i,\epsilon}$ for every
  $i=1,\ldots,r$.   
Then 
$$
p(F;W)=0 \implies p(F;W')=0 \implies p(F;W'_m)=0,
%\implies p(F;W_m)=0,
$$
for every graph $F$, and it follows from \refL{lem:hered} that 
 $W'_m \in \widehat{\QQ}$.

By construction and \eqref{b80}, if $i \neq j$,
\begin{equation}
\label{eq:b8}
\lambda\left(\left\{(x,y) \in I_i \times I_j : 
W'_m(x,y)=\tfrac{1}{2} \right\}\right) \ge \lrpar{1-\tfrac1m} |I_i|\cdot |I_j|.
\end{equation}
Regard $W'_m|_{I_i \times I_i}$ as a graphon (rescaling $I_i$ to $[0,1]$), and
choose a subsequence of $W'_m$ such that $W'_m|_{I_i \times I_i}$
converges for each $i$ to some limit $U_i$. 
It then follows from (\ref{eq:b8})
that $W'_m \to W^*$ along the subsequence, 
where $W^*=\frac{1}{2}$ on $I_i \times I_j$ when
$i \neq j$ and $W^*=U_i$ on $I_i \times I_i$. 
Thus $W^* \in \cqoo$. % \cap R_r$.

Furthermore, 
for any $n$, if $N$ is large enough, then by Ramsey's theorem, 
the random graph $G(N,U_i)$ contains a copy of $K_n$ or its complement
$\overline{K_n}$, and hence $G(n,U_i)$ equals $K_n$ or $\overline{K_n}$ with
positive probability. 
It follows easily, using \eqref{gnw}, that either $p(K_n;U_i)>0$ for all $n$
or $p(\overline{K_n};U_i)>0$ for all $n$ (or both). In the first case 
we may modify $W^*$ by replacing 
$U_i$ by the constant 1 on $I_i\times I_i$, and in the second case we may
instead replace $U_i$ by 0; using \refL{lem:hered} and \eqref{gnw}, it is
easily seen that the modification still belongs to $\cqoo$. 
Doing such a modification for each $i$, we obtain (possibly after a 
rearrangement of the intervals $I_i$) a graphon $\wwrs\in\cqoo$ for some $s$.
%with $0\le s\le r$.
\end{proof}

\begin{remark}
  We allow $r=1$ in \refL{lem:withclique}. 
Since no graphon is $K_1$-free, it then says that if $\cQ$ is any (infinite)
hereditary class of graphs, then $\cqoo$ contains some graphon in $R_1$,
\ie, some random-free graphon, and more precisely, 
at least one of the constant graphons $\wwx{1,0}=0$
and $\wwx{1,1}=1$. 
(The proof above is valid, but may be much simplified in this case. 
See also \cite{SJ255}.)
\end{remark}

\begin{lemma}
\label{lem:entropy_cliquefree}
Let $1\le r<\infty$. 
If the randomness support graphon of $W$ is $K_{r+1}$-free, then
$\Ent(W) \le 1-\frac{1}{r}$ with equality if and only if $W \in R_r$
up to equivalence.
\end{lemma}
\begin{proof}

Let $W'(x,y)$ by the randomness support graphon of $W$.
By assumption, $W'$ 
is $K_{r+1}$-free, 
so by Theorem~\ref{thm:ErdosSimonovits}, we have $\iint W'\le 1-\frac{1}{r}$.
Moreover since $h(x) \le 1$ always, and $h(x)=0$ when $x \in \{0,1\}$, we have
$h(W(x,y)) \le W'(x,y)$ for all $x,y$ and thus
$$\Ent(W) = \iint h(W) \le \iint W' \le 1 - \frac{1}{r},$$
with equality holding only if $W=\frac{1}{2}$ almost everywhere on its
randomness support and 
$W'$ is equivalent to $W_{K_r}$, which implies that 
$W$ is equivalent to a graphon in $R_r$.
(For a rigorous proof of the latter fact, we may use 
\cite[Lemma 23]{Pikhurko10}, which implies that $W'$ \aex{} equals $W_{K_r}$ 
up to a measure-preserving bijection of $\oi$.)

Conversely, it is obvious that $\Ent(W)=1-\frac{1}{r}$ for every $W \in
R_r$,
see \eqref{eq:TuranGraphons}.
\end{proof}

\begin{proof}[Proof of \refT{thm:TR}]
 Let $r \le \infty$ be defined by (\ref{eq:define_r}). If $r<\infty$, then every $W \in
\widehat{\QQ}$ has a randomness support graphon that is $K_{r+1}$-free and thus
$\Ent(W) \le 1-\frac{1}{r}$ by Lemma~\ref{lem:entropy_cliquefree}. 
Moreover,  there exists $W \in
\widehat{\QQ}$ with a randomness support graphon that is not $K_r$-free and thus
by Lemma~\ref{lem:withclique} there exists $W' \in \widehat{\QQ}$ with $W'
\in R_r$, which by \eqref{eq:TuranGraphons} implies
$\Ent(W')=1-1/r$.
Hence,
\begin{equation}
  \label{hz}
%\Ent(\widehat{\QQ}) := 
\max_{W \in \widehat{\QQ}} \Ent(W)=1-\frac{1}{r}, 
\end{equation}
and by Lemma~\ref{lem:entropy_cliquefree} the maximum is attained only for
$W \in \widehat{\QQ} \cap R_r$. 
Thus $\widehat{\QQ}^* = \widehat{\QQ} \cap R_r$.
Furthermore, by \refL{lem:withclique}, some $\wwrs\in\cqoo$, and
$\wwx{t,s}\notin\cqoo$ for $t>r$ since $\Ent(\wwx{t,s})=1-1/t>1-1/r$; hence
\eqref{eq:r2} holds.

If $r=\infty$, there is for every $s<\infty$ a graphon in $\widehat{\QQ}$
whose randomness support graphon is not $K_s$-free and thus
Lemma~\ref{lem:withclique} shows that 
there exists a graphon $W_s \in \widehat{\QQ} \cap R_s$. 
But then $W_s$ converges to the constant graphon $\frac{1}{2}$ in cut norm
(and even in $L^1$)  
as $s\to\infty$. Thus the constant graphon
$\frac{1}{2}$ belongs to $\widehat{\QQ}$. 
Since $\Ent(\frac12)=1$, it follows that 
$\max_{W \in \widehat{\QQ}} \Ent(W)=1$, \ie, 
\eqref{hz} holds in the case $r=\infty$ too.
Moreover, $\frac12$ is the only graphon with entropy 1, see \refR{R1},
and thus $\cqoo^*=\set{\frac12}=R_\infty$.

We have shown \eqref{hz} for any $r\le\infty$, and thus
(\ref{eq:thm_TR}) follows by Theorem~\ref{thm:T4}.

By \eqref{hz}, $r=1$ if and only if $\Ent(W)=0$ for every $W\in\cqoo$, \ie,
if and only if every $W\in\cqoo$ is random-free, which by definition 
means that $\cQ$ is random-free.

Finally, if $r=\infty$, then we have established that $\frac{1}{2} \in
\widehat{\QQ}$. Since $p(F;\frac12)>0$ for every $F$ by \eqref{phw}, 
Lemma~\ref{lem:hered}  shows that $\QQ$ is the class of all graphs.
Hence, every graphon belongs to $\cqoo$, so \eqref{eq:r2} holds  
trivially in this case too.
\end{proof}

%\bibliographystyle{alpha}
%\bibliography{entropy}

\newcommand{\etalchar}[1]{$^{#1}$}

\end{document}